\documentclass[a4paper,12pt]{amsart}

\usepackage[utf8]{inputenc}
\usepackage[T1]{fontenc}
\usepackage[english]{babel}

\usepackage[margin=7em]{geometry}

\usepackage{amssymb, amsfonts}

\usepackage{bbm}

\usepackage{enumerate}

\numberwithin{equation}{section}

\newtheorem{theoremcounter}{theoremcounter}[section]

\newtheorem{conjecture}[theoremcounter]{Conjecture}
\newtheorem{corollary}[theoremcounter]{Corollary}
\newtheorem{definition}[theoremcounter]{Definition}

\newtheorem{lemma}[theoremcounter]{Lemma}
\newtheorem{proposition}[theoremcounter]{Proposition}
\newtheorem{remark}[theoremcounter]{Remark}

\newtheorem{theorem}[theoremcounter]{Theorem}

\renewcommand{\frak}{\ensuremath{\mathfrak}}
\newcommand{\cal}{\ensuremath{\mathcal}}
\newcommand{\bboard}{\ensuremath{\mathbb}}

\newcommand{\frake}{\ensuremath{\frak{e}}}

\newcommand{\cF}{\ensuremath{\cal{F}}}

\newcommand{\cL}{\ensuremath{\cal{L}}}

\newcommand{\cT}{\ensuremath{\cal{T}}}

\newcommand{\bbH}{\ensuremath{\bboard H}}

\newcommand{\rmH}{\ensuremath{\mathrm{H}}}
\newcommand{\rmJ}{\ensuremath{\mathrm{J}}}

\newcommand{\rmM}{\ensuremath{\mathrm{M}}}

\newcommand{\rmt}{\ensuremath{\mathrm{t}}}

\newcommand{\ZZ}{\ensuremath{\mathbb{Z}}}
\newcommand{\QQ}{\ensuremath{\mathbb{Q}}}
\newcommand{\RR}{\ensuremath{\mathbb{R}}}
\newcommand{\CC}{\ensuremath{\mathbb{C}}}

\renewcommand{\pmod}[1]{\ensuremath{\;(\mathrm{mod}\, #1)}}

\newcommand{\Mat}[1]{\ensuremath{\mathrm{Mat}_{#1}}}
\newcommand{\MatT}[1]{\ensuremath{\mathrm{Mat}^\rmt_{#1}}}

\newcommand{\GL}[1]{\ensuremath{\mathrm{GL}_{#1}}}
\newcommand{\SL}[1]{\ensuremath{\mathrm{SL}_{#1}}}
\newcommand{\Mp}[1]{\ensuremath{\mathrm{Mp}_{#1}}}
\newcommand{\Sp}[1]{\ensuremath{\mathrm{Sp}_{#1}}}
\newcommand{\Orth}[1]{\ensuremath{\mathrm{O}_{#1}}}

\newcommand{\rT}{\ensuremath{\hspace{0.12em}{}^\rmt\hspace{-0.09em}}}

\newcommand{\tr}{\ensuremath{\mathrm{tr}}}

\newcommand{\rk}{\ensuremath{\mathop{\mathrm{rk}}}}

\newcommand{\slashdiv}{\ensuremath{\mathop{/}}}

\newcommand{\lspan}{\ensuremath{\mathop{\mathrm{span}}}}

\newcommand{\HS}{\mathbb{H}}

\newcommand{\td}{\tilde}

\newcommand{\disc}{\ensuremath{{\rm disc}}}
\newcommand{\tGamma}{\ensuremath{\widetilde{\Gamma}}}
\newcommand{\rd}{\ensuremath{\mathop{\rm rd}}}
\newcommand{\md}{\ensuremath{\mathop{\rm md}}}
\newcommand{\CH}{\ensuremath{\mathrm{CH}}}

\begin{document}

\title{Spans of special cycles of codimension less than~$5$}

\author{Martin Raum}
\address{ETH, Dept. Mathematics, Rämistraße 101, CH-8092, Zürich, Switzerland}
\email{martin.raum@math.ethz.ch}
\urladdr{http://www.raum-brothers.eu/martin/}
\thanks{The author is supported by the ETH Zurich Postdoctoral Fellowship Program and by the Marie Curie Actions for People COFUND Program.}


\begin{abstract}
We show that the span of special cycles in the $r$th Chow group of a Shimura variety of orthogonal type is finite dimensional, if $r < 5$.  As our main tool, we develop the theory of Jacobi forms with rational index $M \in \Mat{N}(\QQ)$.
\end{abstract}

\maketitle

\section{Introduction}
\label{sec:introduction}

In 1956, Chow introduced cohomology groups attached to any variety~$X$~\cite{Ch56}, which, by now, are called Chow groups and are denoted by~$\CH^r(X)$.  They are constructed as formal sums up to rational equivalence of subvarieties of constant codimension~$r$.  The first Chow group equals the Picard group, which in interesting cases classifies holomorphic vector bundles up to isomorphisms.  While Picard groups are difficult to study, our knowledge about Chow groups is even more restricted.  Even their rank stays mysterious in most cases.

In 1998, Kudla suggested to consider ``special cycles''~\cite{Ku97}, that he constructed on Shimura varieties of orthogonal type.  These cycles, he conjectured, behave much better than general cycles.  As an instance of this rather philosophical statement, he proposed that the generating function of special cycles of codimension~$r$ is modular.  In particular, the span of special cycles, he predicted, is finite dimensional.
\begin{conjecture}
Given a Shimura variety~$X_\Gamma$ of orthogonal type, the generating function of codimension~$r$ cycles on~$X_\Gamma$ is a Siegel modular form of degree~$r$.

In particular, the span of special cycles in $\CH^r(X_\Gamma)_\CC$ has finite dimension.
\end{conjecture}

In parallel, to Kudla's work, Borcherds finished his studies of product expansions of automorphic forms~\cite{Bo98}.  As a result, he was able to resolve completely Kulda's conjecture in the case of~$r = 1$.
\begin{theorem}[{Borcherds~\cite{Bo99, Bo00}}]
For $r = 1$, Kudla's conjecture is true.
\end{theorem}
\noindent Zhang dedicated his thesis to the above conjecture, and he was able to obtain partial results.
\begin{theorem}[Zhang~\cite{Zh09}]
If $r = 2$, the span of special cycles in $\CH^r(X_\Gamma)$ has finite dimension.
\end{theorem}
In this paper, we extend the range for which this statement holds.
\begin{theorem}
If $r < 5$, the span of special cycles in $\CH^r(X_\Gamma)$ has finite dimension.
\end{theorem}
\noindent A more detailed statement can be found in Theorem~\ref{thm:main-theorem-finite-span-of-cycles}.

Zhang used an ad hoc method to prove his result, and on page~35 he stated that it ``does not generalize to higher codimension cycles''.  We show that similar ideas allow to cover all cases $r < 5$.  Zhang employed induction and standard vanishing results for elliptic modular forms of higher level.  In his proof, these elliptic modular forms arose form Jacobi forms, which he did not consider directly.  The key to our result is not to pass to elliptic modular forms, but study Jacobi forms directly.  This leads us to develop a theory of vector valued Jacobi forms whose indices are matrices with rational entries.
\vspace*{1ex}

Jacobi forms show up in many important places, including the proof of the Saito-Kurokawa conjecture~\cite{An79, Ma79a, Ma79b, Ma79c, Za81} and modern string theory~\cite{DG07, CD12, DMZ11}.  They were originally defined as functions on $\HS \times \CC$ with an index~$m \in \ZZ$ attached to them.  Later, this notion was extended to functions on $\HS \times \CC^N$ with matrix index~$M \in \Mat{N}(\ZZ)$~\cite{Zi89, Sk08}.  A further generalization it shows up naturally while studying theta lifts or special cycles.  We extend the theory of Jacobi forms to indices~$M \in \Mat{N}(\QQ)$.  We show that there is an analog of the theta decomposition, which links Jacob forms to vector valued elliptic modular forms.  Theta decompositions of usual Jacobi forms in the most general case is due to Shimura and Ziegler~\cite{Sh78, Zi89}.  We write $\rmJ_{k, M}(\rho)$ for the space of Jacobi form of weight~$k \in \frac{1}{2}\ZZ$, index~$M \in \Mat{N}(\QQ)$ and type~$\rho$, where $M$ is positive definite and symmetric, and $\rho$ is a finite dimensional unitary representation of the metaplectic Jacobi group (see Definition~\ref{def:jacobi-forms}).
\begin{theorem}
\label{thm:main-theorem-theta-decomposition}
For given $M$ and $\rho$ there are functions $\theta_{M, \rho; \nu}$ such that for every $k$ and $\phi \in \rmJ_{k, M}(\rho)$, we have
\begin{gather*}
  \phi(\tau, z)
=
  \sum_\nu h_\nu(\tau) \theta_{M, \rho; \nu}(\tau, z)
\text{.}
\end{gather*}
The components $h_\nu$ form a vector valued elliptic modular form of weight~$k - \frac{N}{2}$.
\end{theorem}
\noindent A more detailed statement and exact expressions for $\theta_{M, \rho; \nu}$ are provided in Theorem~\ref{thm:theta-decomposition}.  The vanishing statement that is crucial to our argument, is given in Section~\ref{ssec:jacobi-forms:vanishing}.

The study of Jacobi forms requires classification of those representations of the full Jacobi group that occur for non-trivial Jacobi forms.  We build standard representations of the metaplectic Jacobi group, that, as we show, occur as tensor factors of every irreducible representation.  As a next step, we construct theta functions attached to each of these standard representations.  This finally allows us to decompose any Jacobi forms.
\vspace*{1ex}

A brief word on notation.  This paper addresses people working on Shimura varieties of orthogonal type and working on Jacobi forms.  Notation in these areas varies slightly, and we have decided to accommodate both as much as possible.  In Section~\ref{sec:jacobi-forms}, which is dedicated to Jacobi forms, we use notation common in this area.  In Section~\ref{sec:special-cycles}, we adopt the notation used by Zhang~\cite{Zh09}.

Section~\ref{sec:modular-forms} contains preliminaries on modular forms and a vanishing result.  In Section~\ref{sec:jacobi-forms}, we build up the theory of Jacobi forms.  This includes theta decomposition, a vanishing theorem, and a proposition on the geometry of lattices.  In Section~\ref{sec:special-cycles}, we combine these results in order to prove the main theorem.

\section{Vector valued modular forms}
\label{sec:modular-forms}

We start by defining vector valued elliptic modular forms, that we will later relate to Jacobi forms.  The Poincar\'e upper half plane is $\HS = \{\tau = u + i v \,:\, v > 0\} \subset \CC$.  We write $q$ for $\exp(2 \pi i\, \tau) = e(\tau)$.  The metaplectic cover $\Mp{2}(\ZZ)$ of $\SL{2}(\ZZ)$ is the preimage of $\SL{2}(\ZZ)$ in $\Mp{2}(\RR)$, the connected double cover of $\SL{2}(\RR)$.  Write $\gamma = \left(\begin{smallmatrix} a & b \\ c & d \end{smallmatrix}\right)$ for a typical element of $\SL{2}(\RR)$.  The elements of $\Mp{2}(\RR)$ can be written as $\big(\gamma,\, \tau \mapsto \sqrt{c \tau + d}\big)$, where the first component is an element of $\SL{2}(\RR)$ and the second is a holomorphic root on $\HS$, that we usually write as~$\omega :\, \HS \rightarrow \CC$.  Since there are two branches of the square root, this yields indeed a double cover of $\SL{2}(\RR)$.  The product in $\Mp{2}(\ZZ)$ is defined as
\begin{gather*}
  (\gamma_1, \omega_1) (\gamma_2, \omega_2)
=
  (\gamma_1 \gamma_2, \omega_1 \circ \gamma_2 \cdot \omega_2)
\text{.}
\end{gather*}
The elements
\begin{gather*}
  T
:=
  \big( \left(\begin{smallmatrix}1 & 1 \\ 0 & 1\end{smallmatrix}\right),\, \tau \mapsto 1 \big)
\quad\text{and}\quad
  S
:=
  \big( \left(\begin{smallmatrix}0 & -1 \\ 1 & 0\end{smallmatrix}\right),\, \tau \mapsto \sqrt{\tau} \big)
\end{gather*}
generate $\Mp{2}(\ZZ)$.  The root in the definition of $S$ is the principal branch, which maps $i$ to $\exp(\frac{1}{2} \pi i)$.

An action of $\SL{2}(\RR)$ (and thus of $\Mp{2}(\ZZ) \twoheadrightarrow \SL{2}(\ZZ)$) on $\HS$ is given by
\begin{gather*}
  \gamma \tau
=
  \frac{a \tau + b}{c \tau + d}
\text{.}
\end{gather*}

Given a representation~$\rho$ of $\Mp{2}(\ZZ)$ with representation space~$V(\rho)$, $k \in \tfrac{1}{2}\ZZ$, and $f :\, \HS \rightarrow V(\rho)$, we define
\begin{gather*}
  \big( f\big|_{k,\rho} \, (\gamma, \omega) \big)\, (\tau)
:=
  \omega(\tau)^{-2 k} \, \rho\big((\gamma, \omega)\big)^{-1} \,
  f\big( \gamma \tau \big)
\end{gather*}
for all $(\gamma, \omega) \in \Mp{2}(\RR)$.
\begin{definition}
Let $(\rho, V(\rho))$ be a finite-dimensional representation of $\Mp{2}(\ZZ)$.  A vector valued modular form of weight~$k \in \tfrac{1}{2}\ZZ$ and type~$\rho$ is a holomorphic function $f \,:\, \HS \rightarrow V(\rho)$ such that the following conditions are satisfied:
\begin{enumerate}[(i)]
\item For all $\gamma \in \Mp{2}(\ZZ)$ we have $f\big|_{k,\rho} \, \gamma = f$.
\item We have $\| f(\tau) \| = O(1)$ as $y \rightarrow \infty$, where $\| \,\cdot\, \|$ is some norm on $V(\rho)$.
\end{enumerate}
\end{definition}
\noindent We write $\rmM_k(\rho)$ for the space of modular forms of weight~$k$ and type~$\rho$.

In this paper, our interest lies in unitary representations~$\rho$.  If $\rho$ is unitary, $\rho(T)$ is diagonalizable and $f \in \rmM_k(\rho)$ has Fourier expansion
\begin{gather*}
  f(\tau)
=
  \sum_{0 \le m \in \QQ} c(\phi; m) \, q^m
\end{gather*}
with $c(\phi; m) \in V(\rho)$.

Our chief interest in vector valued modular forms originates in the following vanishing result.
\begin{proposition}
\label{prop:vanishing-of-vector-valued-forms}
Suppose that $\phi \in M_{k}(\rho)$ and $c(\phi; m) = 0$ for all $m < \frac{k}{12} + 1$.  Then we have $\phi = 0$.
\end{proposition}
\begin{proof}
Let $l = \lfloor \frac{k}{12} \rfloor + 1 > \frac{k}{12}$.  Given $\phi$ as in the statement, we have  $\Delta^{-l} \phi \in M_{k - 12 l}(\rho)$, where $\Delta$ is the unique cusp form of weight~$12$ and trivial type.  If $\phi \ne 0$, then we obtain a non-zero modular form of negative weight, which cannot be.
\end{proof}

\section{Jacobi forms}
\label{sec:jacobi-forms}

We now prepare for defining vector valued Jacobi forms.  Given $0 < N \in \ZZ$, let
\begin{gather*}
  \HS^{\rmJ(N)}
:=
  \HS \times \CC^N
\end{gather*}
be the Jacobi upper half space.  The metaplectic cover of the centrally extended Jacobi group is
\begin{gather*}
  \tGamma^{\rmJ(N)}
:=
  \Mp{2}(\ZZ) \ltimes \Mat{N, 2}(\ZZ) \widetilde{\times} \MatT{N}(\ZZ)
\text{.}
\end{gather*}
We denote the group of symmetric matrices by $\MatT{N}(\,\cdot\,)$, and accordingly we write $\rT (\,\cdot\,)$ for the transpose of a vector or matrix.

Typical elements of the second component are written $(\lambda, \mu)$, $\lambda, \mu \in \ZZ^N$.  The product in $\tGamma^{\rmJ(N)}$ is given by
\begin{multline*}
  \big((\gamma_1, \omega_1), (\lambda_1, \mu_1), \kappa_1 \big)
  \big((\gamma_2, \omega_2), (\lambda_2, \mu_2), \kappa_2 \big)
\\
=
  \big( (\gamma_1, \omega_1) (\gamma_2, \omega_2),\;
        (\lambda_1, \mu_1) \gamma_2 + (\lambda_2, \mu_2),\;
        \kappa_1 + \kappa_2 + \lambda_1 \rT \mu_2 - \mu_1 \rT \lambda_2
        \big)
\text{.}
\end{multline*}
In particular, the semidirect product is defined via the natural right action of $\SL{2}(\ZZ)$ on $\Mat{N, 2}(\ZZ)$.  As before, we usually write $\gamma = \left(\begin{smallmatrix}a & b \\ c & d\end{smallmatrix}\right)$ for elements of $\SL{2}(\ZZ)$.  We frequently use the notation
\begin{gather*}
  [\lambda, \mu, \kappa]^\rmJ
=
  \big( (I_2, \tau \mapsto 1), (\lambda, \mu), \kappa \big)
\end{gather*}
for elements of the Heisenberg-like group
\begin{gather*}
  \rmH (\tGamma^{\rmJ(N)} )
:=
  \Mat{N, 2}(\ZZ) \widetilde{\times} \Mat{N}(\ZZ)
\subset
  \tGamma^{\rmJ(N)}
\text{.}
\end{gather*}
In most cases, elements~$\gamma^\rmJ \in \tGamma^{\rmJ(N)}$ are given as tuples $((\gamma, \omega), (\lambda, \mu), \kappa)$.  In order to lighten notation, we occasionally switch the position of~$\omega$, writing
\begin{gather*}
  (\gamma^\rmJ, \omega)
=
  \big( (\gamma, (\lambda, \mu), \kappa), \omega \big)
\end{gather*}
for $( (\gamma, \omega), (\lambda, \mu), \kappa ) \in \tGamma^{\rmJ(N)}$.

There is an action $\tGamma^{\rmJ(N)} \times \HS^{\rmJ(N)} \rightarrow \HS^{\rmJ(N)}$ is given by
\begin{gather}
  (\gamma^\rmJ, \omega) (\tau, z)
=
  \Big( \frac{a \tau + b}{c \tau + d}, \frac{z + \lambda \tau + \mu}{c \tau + d} \Big)
\text{.}
\end{gather}

We next give a family of actions on functions on $\HS^{\rmJ(N)}$.  Let $\tr(\,\cdot\,)$ denote the trace of a matrix.  For convenience, we set
\begin{gather*}
  \alpha_M\big( (\gamma, (\lambda, \mu), \kappa); \tau, z\big)
=
  e\Big( \frac{- cM[z + \lambda \tau + \mu]}{c \tau + d} 
         - M[\lambda] \tau - 2 \rT \lambda M z
         - \tfrac{1}{2} \tr(M \kappa) \Big)
\end{gather*}
for any $M \in \MatT{N}(\QQ)$.  Let $\rho$ be a finite dimensional representation of $\tGamma^{\rmJ(N)}$ with representation space~$V(\rho)$.  Fix a weight $k \in \frac{1}{2} \ZZ$ and a (Jacobi) index $M \in \MatT{N}(\ZZ)$ which is positive definite and has even diagonal entries.  Note that in the case $N = 1$, it is more common to use indices $m \in \QQ$.  These notions are releated by $M = (2 m)$.   We abbreviate $z^\tr M z$ ($z \in \CC^N$) by $M[z]$.  Recall that $e(x) = \exp(2 \pi i\, x)$.  Given $\phi :\, \bbH^{\rmJ(N)} \rightarrow V(\rho)$, we define
\begin{gather*}
  \big( \phi \big|_{k, M, \rho}\, (\gamma^\rmJ, \omega) \big) (\tau, z)
=
  \omega(c \tau + d)^{-2k} \alpha_M(\gamma^\rmJ; \tau, z)^{-1}
  \;
  \rho\big( (\gamma^\rmJ, \omega) \big)^{-1}
  \;
  \phi\big(\gamma^\rmJ (\tau, z) \big)
\text{.}
\end{gather*}

The next definition generalizes a definition made in~\cite{Zi89} to vector valued Jacobi forms which have non-integral indices.
\begin{definition}
\label{def:jacobi-forms}
Suppose that $k \in \frac{1}{2}\ZZ$, $M \in \MatT{N}(\QQ)$, and $(\rho, V(\rho))$ is a finite dimensional representation of $\tGamma^{\rmJ(N)}$.  A holomorphic function $\phi :\, \HS^{\rmJ(N)} \rightarrow V(\rho)$ is called a Jacobi form of weight~$k$, index~$M$, and type~$\rho$ if
\begin{enumerate}[(i)]
\item We have $\phi \big|_{k, M, \rho}\, (\gamma^\rmJ, \omega) = \phi$ for all $(\gamma^\rmJ, \omega) \in \tGamma^{\rmJ(N)}$.
\item We have $\| \phi(\tau, \alpha \tau + \beta) \| = O(1)$ for all $\alpha, \beta \in \QQ^N$ and some norm $\|\,\cdot\,\|$ on $V(\rho)$.
\end{enumerate}
\end{definition}
\noindent We denote the space of Jacobi forms of weight~$k$, index~$M$, and type~$\rho$ by $\rmJ_{k, M}(\rho)$.  Throughout this section, we assume that $k$, $M$, and $\rho$ satisfy the assumptions in the previous definition.

We write $\zeta^r$ for $e(\rT r z)$.  If $\rho$ is unitary, the second condition in Definition~\ref{def:jacobi-forms} is equivalent to
\begin{gather*}
  \phi(\tau, z)
=
  \sum_{m \in \QQ,\, r \in \QQ^N} c(\phi; m, r)\, q^m \zeta^r
\text{,}
\end{gather*}
for Fourier coefficients $c(\phi; m, r)$ that vanish if $2 \det(M) n - M^\#[r] < 0$ or $m < 0$.  The adjunct of $M$ is denoted by $M^\#$.  Because Jacobi forms are invariant under the action of $[\lambda, \mu, 0]^\rmJ$ ($\lambda, \mu \in \ZZ^N$), the Fourier coefficients $c(\phi; m, r)$ satisfy the following relations, which are connected to the statement of Theorem~\ref{thm:theta-decomposition}.
\begin{align}
\label{eq:higherjacobiforms_fourier_relation1}
  c(\phi; m, r)
&=
  c(\phi; m + \tfrac{1}{2} M[\lambda] + r^\tr \lambda, r + M \lambda)
\\[4pt]
\label{eq:higherjacobiforms_fourier_relation2}
  c(\phi; m, -r)
&=
  \omega(-1)^{2 k} \rho((-I_2, \omega)) \, c(\phi; m, r)
\text{.}
\end{align}

Note that we have $\rmJ_{k, M}(\rho_1 \oplus \rho_2) = \rmJ_{k, M}(\rho_1) \oplus \rmJ_{k, M}(\rho_2)$.  This allows us to reduce most considerations to the case of irreducible $\rho$.

Among experts it is well-known that on can restrict to positive definite~$M$.  The next proposition, however, is not in the literature.  We do not need it, but it seems appropriate to justify our future assumption $M > 0$.
\begin{proposition}
Suppose that $M$ is not positive semidefinite.  Then $\rmJ_{k, M}(\rho) = \{0\}$.

If $M$ is positive semidefinite, but not positive definite , then there is $s \in \MatT{N, N - 1}(\ZZ)$ such that 
\begin{gather*}
  \rmJ_{k, M}(\rho)
\cong
  \rmJ_{k, M[s]}(\rho)
\quad\text{via}\quad
  \phi(\tau, z)
\mapsto
  \phi(\tau, s z')
\text{,}
\end{gather*}
where $z' \in \CC^{N - 1}$.
\end{proposition}
\begin{proof}
Assume that $\rmJ_{k, M}(\rho) \ne \{ 0 \}$.  We first show that $M$ is positive semidefinite.  Suppose there was $\lambda \in \ZZ^N$ such that $M[\lambda] < 0$.  Fix some $0 \ne \phi \in \rmJ_{k, M}(\rho)$.  By~\eqref{eq:higherjacobiforms_fourier_relation1}, for any $m \in \QQ$, $r \in \QQ^N$, and $l \in \ZZ$, we have
\begin{gather*}
  c(\phi; m, r)
=
  c(\phi; m + \tfrac{l^2}{2} M[\lambda] + l r^\tr \lambda, r + l M \lambda)
\text{.}
\end{gather*}
By choosing $l$ large enough we can force the second argument on the right hand side to be negative, which shows that $c(\phi; m, r) = 0$.  Since $m$ and $r$ were arbitrary, this shows that $\phi = 0$, contradicting our choice.

We are reduced to the case of positive semidefinite~$M$.  If $M$ is not positive definite, there is a primitive vector $\lambda \in \ZZ^N$ with $M[\lambda] = 0$.  Fix $g \in \GL{N}(\ZZ)$ with last column~$\lambda$.  Then
\begin{gather*}
  \rmJ_{k, M}
\cong
  \rmJ_{k, M[g]}
\quad\text{via}\quad
  \phi(\tau, z)
\mapsto
  \phi(\tau, g z)
\text{.}
\end{gather*}
It thus suffices to treat the case of $M_{N, N} = 0$.  Since $M$ is positive semidefinite we have $M_{N, i} = 0$ for all $1 \le i \le N$.

Fix $\tau$ and $z_1, \ldots, z_{N - 1}$, and let $z_N$ vary.  Applying the transformation law for Jacobi forms, we see that
\begin{gather*}
  \phi\big( \tau, (z_1, \ldots, z_{N - 1}, z_N + \lambda \tau + \mu) \big)
=
  \phi\big( \tau, (z_1, \ldots, z_N) \big)
\end{gather*}
for all $\lambda, \mu \in \ZZ$.  That is
\begin{gather*}
  z_N
\mapsto
  \phi\big( \tau, (z_1, \ldots, z_N) \big)
\end{gather*}
is holomorphic and bounded.  Hence it is constant.  This shows that we can choose $s = (\delta_{i,j})_{1 \le i \le N,\, 1 \le j \le N - 1}$, where $\delta$ is the Kronecker delta.

\end{proof}

\subsection{Representations of the Heisenberg-like part of $\tGamma^{\rmJ(N)}$.}
\label{ssec:heisenberg-representations}

Representations of the Heisenberg-like group $\rmH (\tGamma^{\rmJ(N)} ) \subset \tGamma^{\rmJ(N)}$ that occur for non-zero Jacobi forms can be classified with little effort.  Here and throughout Section~\ref{ssec:heisenberg-representations}, we assume that $\pi$ is an irreducible, finite dimensional and unitary representation of $\rmH (\tGamma^{\rmJ(N)} )$.

Since $\Mat{N}(\ZZ)$ it central in $\rmH( \tGamma^{(\rmJ(N)} )$, there is $M_\pi \in \Mat{N}(\QQ)$ such that $\pi( [0, 0, -\kappa]^\rmJ )$ acts as $e( \tr(\kappa M_\pi) )$.  Note that $M_\pi$ is only well-defined up to elements in $\Mat{N}(\ZZ)$.  Throughout the paper we assume that $M_\pi$ is symmetric.  We justify this assumption in Proposition~\ref{prop:kappa-eigenvalues}.  Since $M_\pi$ can be freely modified by elements in $\Mat{N}(\ZZ)$, we can and will also assume that $M_\pi$ is positive definite.  We say that $\pi$ has central character $e(\tr(M_\pi \,\cdot\,))$.

Given $0 < M \in \MatT{N}(\QQ)$ fix a Smith normal form $M = U D V$.  Decompose $D$ as $D_\ZZ D^\ZZ$, where $D_\ZZ$ is the diagonal matrix whose entries are the numerators of the entries of~$D$.  Set
\begin{gather}
\label{eq:MZZ-definition}
  M_\ZZ
=
  U D_\ZZ V
\quad\text{and}\quad
  M^\ZZ
=
  U D^\ZZ V
\text{.}
\end{gather}
We can view $M_\ZZ$ and $M^\ZZ$ as the integral and fractional part of $M$, respectively.
\begin{lemma}
\label{la:integral-M-and-its-inverse}
We have
\begin{alignat*}{2}
  M_\ZZ \ZZ^N
&=
  M \ZZ^N \cap \ZZ^N
\text{,}\quad
&
  M_\ZZ^{-1} \ZZ^N
&=
  M^{-1} \ZZ^N + \ZZ^N
\text{,}
\\
  M^\ZZ \ZZ^N
&=
  M \ZZ^N + \ZZ^N
\text{,}
&
  (M^\ZZ)^{-1} \ZZ^N
&=
  M^{-1} \ZZ^N \cap \ZZ^N
\text{.}
\end{alignat*}
\end{lemma}
\begin{proof}
Given the Smith normal form of $M = U D V$, write $M_\ZZ = U D_\ZZ V$ as in~(\ref{eq:MZZ-definition}).  We have
\begin{align*}
  M_\ZZ \ZZ^N
&=
  U D_\ZZ \ZZ^N
=
  U \big( D \ZZ^N \cap \ZZ^N \big)
=
  M \ZZ^N \cap \ZZ^N
\text{,}\quad\text{and}
\\[3pt]
  M_\ZZ^{-1} \ZZ^N
&=
  V^{-1} D_\ZZ^{-1} \ZZ^N
=
  V^{-1} \big( D^{-1} \ZZ^N + \ZZ^N \big)
=
  M^{-1} \ZZ^N + \ZZ^N
\text{.}
\end{align*}
This proves the first and second equality.  The remaining assertions can be proved in the same way.
\end{proof}
\begin{lemma}
\label{la:integral-evaluation-of-M}
For $\lambda \in M^{-1} \ZZ^N \cap \ZZ^N$, we have $M[\lambda] \in \ZZ$.
\end{lemma}
\begin{proof}
We have $\lambda \in (M^{\ZZ})^{-1} \ZZ^N$, and hence the claim follows from
\begin{gather*}
  M\big[ (M^{\ZZ})^{-1} \big]
=
  (U D V) \big[ V^{-1} (D^\ZZ)^{-1} U^{-1} \big]
=
  \rT U^{-1} (D^\ZZ)^{-1} \rT V^{-1} U \big( D (D^\ZZ)^{-1} \big) U^{-1}
\text{,}
\end{gather*}
where $U$ and $V$ were used to define $M^\ZZ$ in~\eqref{eq:MZZ-definition}.  The fact $D (D^\ZZ)^{-1} \in \Mat{N}(\ZZ)$ shows that
\begin{gather*}
  M \big[ (M^\ZZ)^{-1} \ZZ^N \big]
=
  M \big[ (M^\ZZ)^{-1} \big] [\ZZ^N]
\subseteq
  \ZZ
\text{.}
\end{gather*}
\end{proof}

For $\alpha, \beta \in \QQ^N$, define representations $\pi_{M, \alpha, \beta}$ of $\rmH( \tGamma^{(\rmJ(N)} )$ as follows.  The central character of $\pi_{M, \alpha, \beta}$ is $e(\frac{1}{2} \tr(M \,\cdot\,))$.  Its representation space is
\begin{gather}
\label{eq:pi-Mab-base-space}
  V_{M, \alpha, \beta}
=
  \CC\big[ (M \ZZ^N + \ZZ^N) \slashdiv \ZZ^N \big]
=
  \lspan\big( \frake_r \,:\, r \in (M \ZZ^N + \ZZ^N) \slashdiv \ZZ^N \big)
\text{.}
\end{gather}
The action on $\frake_r$ is given by
\begin{align*}
  \pi( [\lambda, 0, 0]^\rmJ )\, \frake_r
&=
  e\big( \rT \alpha \lambda  \big)\,
  \frake_{r + M^\ZZ \lambda}
\text{,}
\\[3pt]
  \pi( [0, \mu, 0]^\rmJ )\, \frake_r
&=
  e\big( \rT (\beta + r) \mu \big)\,
  \frake_r
\text{.}
\end{align*}
\begin{proposition}
\label{prop:irreps-of-heisenberg-part-predefined}
For $0 < M \in \MatT{N}(\QQ)$ and $\alpha, \beta \in \QQ^N$, $\pi_{M, \alpha, \beta}$ is an irreducible representation of $\rmH( \tGamma^{\rmJ(N)} )$.  Two representations $\pi_{M_1, \alpha_1, \beta_1}$ and $\pi_{M_2, \alpha_2, \beta_2}$ are isomorphic if and only if
\begin{gather*}
  M_1 - M_2 \in \MatT{N}(\ZZ)
\text{,}\quad
  \alpha_1 - \alpha_2,\, \beta_1 - \beta_2 \in M \ZZ^N + \ZZ^N
\text{.}
\end{gather*}
\end{proposition}
\begin{proof}
Clearly, the action of both $[\lambda, 0, 0]^\rmJ$ and $[0, \mu, 0]^\rmJ$ on $\frake_r$ only depends on $r$ modulo~$\ZZ^N$.  We have to check that
\begin{gather*}
  \pi([\lambda, 0, 0]^\rmJ) \pi([0, \mu, 0, 0]^\rmJ)
  \pi([0, 0, 2 \lambda \rT \mu ]^\rmJ)\,
  \frake_r
=
  e(\rT \mu M \lambda)
  e(\rT \alpha \lambda )
  e\big( \rT (\beta + r) \mu  \big) \,
  \frake_{r + M^\ZZ \lambda}
\end{gather*}
equals
\begin{gather*}
  \pi([0, \mu, 0]^\rmJ) \pi([\lambda, 0, 0]^\rmJ)\,
  \frake_r
=
  e(\rT \alpha \lambda )
  e\big(\rT (\beta + r + M \lambda) \mu \big)\,
  \frake_{r + M^\ZZ \lambda}
\text{,}
\end{gather*}
which is obvious.  This show that $\pi_{M, \alpha, \beta}$ is a representation.

Observe that the set of all $[0, \mu, 0]^\rmJ$ ($\mu \in \ZZ^N$) forms a commutative subgroup of $\rmH( \tGamma^{\rmJ(N)} )$.  The system of eigenvalues of $\frake_r$ under the action of $[0, \mu, 0]^\rmJ$ is different for distinct $r \in (M \ZZ^N + \ZZ^N) \slashdiv \ZZ^N$.  Further, the corresponding eigenspaces are permuted transitively by the action of $[\lambda, 0, 0]^\rmJ$ ($\lambda \in \ZZ^N$).  Thus we see that $\pi_{M, \alpha, \beta}$ is irreducible.

In the remainder of this proof, we will use Lemma~\ref{la:integral-M-and-its-inverse} without further reference.  Let $\pi_{M_1, \alpha_1, \beta_1}$ and $\pi_{M_2, \alpha_2, \beta_2}$ by isomorphic representations.  By checking central characters, we find that $M_1 - M_2 \in \MatT{N}(\ZZ)$.  The subgroup of all $[ 0, \mu, 0]^\rmJ$ ($\mu \in M^{-1} \ZZ^N \cap \ZZ^N$) acts by $e(\rT \beta_1 \mu)$ and $e(\rT \beta_2 \mu)$ on $V_{M_1, \alpha_1, \beta_2}$ and $V_{M_2, \alpha_2, \beta_2}$, respectively.  Consequently, we have $\rT \mu (\beta_1 - \beta_2) \in \ZZ$ for all $\mu \in M^{-1}\ZZ^N \cap \ZZ^N$.  We conclude that $\beta_1 - \beta_2 \in M \ZZ^N + \ZZ^N$.  Finally, $[\lambda, 0, 0]^\rmJ$ ($\lambda \in M^{-1} \ZZ^N \cap \ZZ^N$) acts by scalars and the attached eigenvalues are $e(\rT \alpha_1 \lambda)$ and $e(\rT \alpha_2 \lambda)$.  In analogy to the previous considerations, we find that $\rT \lambda (\alpha_1 - \alpha_2) \in \ZZ$ for all $\lambda \in M^{-1} \ZZ^N \cap \ZZ^N$.  Hence $\alpha_1 - \alpha_2 \in M \ZZ^N + \ZZ^N$.

Conversely, suppose that
\begin{gather*}
  M_1 - M_2 \in \MatT{N}(\ZZ)
\text{,}\quad
  \alpha_1 - \alpha_2,\, \beta_1 - \beta_2 \in M \ZZ^N + \ZZ^N
\text{.}
\end{gather*}
We have to show that $\pi_{M_1, \alpha_1, \beta_1} \cong \pi_{M_2, \alpha_2, \beta_2}$.  First of all, the central characters coincide.  Denote the previously chosen basis elements of $V_{M_1, \alpha_1, \beta_1}$ and $V_{M_2, \alpha_2, \beta_2}$ by $\frake_{r, 1}$ and $\frake_{r, 2}$.  Let $\alpha = (M^\ZZ)^{-1} \, ( \alpha_2 - \alpha_1) \in \ZZ^N$ and $\beta =  \beta_1 - \beta_2  \in M \ZZ^N + \ZZ^N$.  We claim that the map
\begin{gather*}
  \iota :\,
  \frake_{r, 1}
\mapsto
  e( \rT \alpha r)\,
  \frake_{r + \beta, 2}
\end{gather*}
intertwines $\pi_{M_1, \alpha_1, \beta_1}$ and $\pi_{M_2, \alpha_2, \beta_2}$.  We have
\begin{multline*}
  \iota\big( \pi_{M_1, \alpha_1, \beta_1}( [0, \mu, 0]^\rmJ )\, \frake_{r,1} \big)
=
  e( \rT \alpha r ) e\big( \rT (\beta_1 + r) \mu \big)\,
  \frake_{r + \beta,2}
\\[3pt]
=
  e\big((\beta_2 + r + \beta) \rT \mu \big) e( \rT \alpha r )\,
  \frake_{r + \beta,2}
=
  \pi_{M_2, \alpha_2, \beta_2}( [0, \mu, 0]^\rmJ )\, \iota( \frake_{r,1} )
\end{multline*}
and
\begin{multline*}
  \iota\big( \pi_{M_1, \alpha_1, \beta_1}( [\lambda, 0, 0]^\rmJ )\, \frake_{r,1} \big)
=
  e\big( \rT \alpha_1 \lambda \big) e\big( \rT \alpha ( r + M^\ZZ \lambda ) \big)\,
  \frake_{r + M^\ZZ \lambda + \beta,2}
\\[3pt]
=
  e\big( \rT \alpha_2 \lambda \big) e( \rT \alpha r )\,
  \frake_{r + M^\ZZ \lambda + \beta,2}
=
  \pi_{M_2, \alpha_2, \beta_2}( [\lambda, 0, 0]^\rmJ )\, \iota( \frake_{r,1} )
\text{.}
\end{multline*}
This completes the proof of Proposition~\ref{prop:irreps-of-heisenberg-part-predefined}.
\end{proof}

\begin{proposition}
\label{prop:irreps-of-heisenberg-part}
Fix $0 < M \in \MatT{N}(\QQ)$.  Any irreducible, finite dimensional, unitary representation of $\rmH( \tGamma^{\rmJ(N)} )$ with central character $e(\frac{1}{2} \tr(M \,\cdot\,))$ is isomorphic to $\pi_{M, \alpha, \beta}$ for some $\alpha, \beta \in \QQ^N$.
\end{proposition}
\begin{proof}
Fix a representation $\pi$ as in the assumption.  Since the set of all $[0, \mu, 0]^\rmJ$ ($\mu \in \ZZ^N$) is a commutative subgroup of $\rmH( \tGamma^{\rmJ(N)} )$, we can find a basis of eigenvectors $v_i$, $1 \le i \le \dim\,\pi$.  Denote the eigenvalue of an eigenvector $v$ under $[0, \mu, 0]^\rmJ$ by $e( \rT \beta(v) \mu )$.  A calculation that is analogous to the computations in the proof of Proposition~\ref{prop:irreps-of-heisenberg-part-predefined} shows that $\pi( [\lambda, 0, 0]^\rmJ )\, v$ is an eigenvector under the action of $[0, \mu, 0]^\rmJ$ if so is $v$.  We have $\beta(\pi( [\lambda, 0, 0]^\rmJ )\, v) = \beta(v) + M^\ZZ \lambda$.

Fix any eigenvector $v$.  Since $\pi$ is unitary and finite dimensional, the orbit $\pi(v) = \big\{ [\lambda,0 , 0]^\rmJ \cdot \lspan{}_\CC (v) \,:\, \lambda \in \ZZ^N \big\}$ is finite.  By inspection of $\beta(\pi( [\lambda, 0, 0]^\rmJ )\,v, \mu)$, we find that the cardinality $\# \pi(v)$ is a multiple of $\# (M \ZZ^N + \ZZ^N) \slashdiv \ZZ^N$.  If it is equal, let $e(\rT \alpha \lambda)$ be the eigenvalue of $v$ under $[\lambda, 0, 0]^\rmJ$ ($\lambda \in M^{-1} \ZZ^N \cap \ZZ^N$).  It is straightforward to verify that $\pi$ is isomorphic to $\pi_{M, \alpha, \beta(v)}$ via the map
\begin{gather*}
  [\lambda, 0, 0]^\rmJ\, v
\mapsto
  e( \rT \alpha \lambda)\, \frake_{M^\ZZ \lambda}
\end{gather*}

In order to complete the proof, we show that $\pi$ is reducible if $\# \pi(v) \ne \# (M \ZZ^N + \ZZ^N) \slashdiv \ZZ^N$.  Define
\begin{gather*}
  B
=
  \ZZ^N \slashdiv \{ \lambda \in \ZZ^N \,:\, [\lambda, 0, 0]^\rmJ\, \lspan{}_\CC(v) = \lspan{}_\CC(v) \}
\text{.}
\end{gather*}
Set ${\td v} = \sum_{\lambda \in B} [\lambda, 0, 0]^\rmJ\, v$, where we fix one choice of representatives of $B$.  By construction of $B$, we see that ${\td v}$ is an eigenvector under all $[0, \mu, 0]^\rmJ$ ($\mu \in \ZZ^N$).  Straightforward computation of the corresponding orbit yields $\# \pi({\td v}) = \# (M \ZZ^N + \ZZ^N) \slashdiv \ZZ^N$, and so ${\td v}$ generates a subrepresentation of $\pi$.  This finishes the proof.
\end{proof}
\vspace{1ex}


\subsection{Representations of $\tGamma^{\rmJ(N)}$}

We have classified representations of $\rmH( \tGamma^{\rmJ(N)} )$, and the next step is to study in which manner they occur as building blocks of representations of $\tGamma^{\rmJ(N)}$.  We henceforth assume that $\rho$ is an irreducible finite dimensional, unitary representation of $\tGamma^{\rmJ(N)}$ with central character $e(\frac{1}{2} \tr( M \,\cdot\,))$.  The restriction of $\rho$ to $\rmH( \tGamma^{\rmJ(N)} )$ is denoted by $\rho |_{\rmH( \tGamma^{\rmJ(N)} )}$.
\begin{lemma}
\label{la:permutation-of-heisenberg-irreps}
Suppose that $\pi$ is a subrepresentation of $\rho |_{\rmH( \tGamma^{\rmJ(N)} )}$.  If $\pi$ is isomorphic to $\pi_{M, \alpha, \beta}$, then for $\gamma \in \Mp{2}(\ZZ)$, $\gamma \pi$ is also a subrepresentation of $\rho |_{\rmH( \tGamma^{\rmJ(N)} )}$.  It is isomorphic to $\pi_{M, \alpha', \beta'}$, where $(\alpha', \beta') = (\alpha, \beta) \rT \gamma$ and $\gamma$ acts via the projection $\Mp{2}(\ZZ) \twoheadrightarrow \SL{2}(\ZZ)$.

Furthermore, the image of $\frake_r$ given in~\eqref{eq:pi-Mab-base-space} under $\gamma$ is uniquely determined by relations in~$\tGamma^{\rmJ(N)}$.
\end{lemma}
\begin{corollary}
Let ${\td \pi}_{M, \alpha, \beta}$ denote the isotypical components of $\rho |_{\rmH( \tGamma^{\rmJ(N)} )}$.  Then we have $\gamma {\td \pi}_{M, \alpha, \beta} = {\td \pi}_{M, \alpha', \beta'}$, where $(\alpha', \beta') = (\alpha, \beta) \gamma$.
\end{corollary}
\begin{proof}[Proof of Lemma~\ref{la:permutation-of-heisenberg-irreps}]
Any irreducible representation of $\rmH( \tGamma^{\rmJ(N)} )$ with central character $e(\frac{1}{2}\tr(M \,\cdot\,))$ has dimension $(M \ZZ^N + \ZZ^N) \slashdiv \ZZ^N$.  By the commutation rules in $\tGamma^{\rmJ(N)}$, $\gamma {\td \pi}_{M, \alpha, \beta}$ is hence irreducible.  Proposition~\ref{prop:irreps-of-heisenberg-part} shows that it is isomorphic to $\pi_{M, \alpha', \beta'}$ for some $\alpha', \beta' \in \QQ^N$.  In order to compute $\alpha'$ and $\beta'$, it suffices to consider the action of $[\lambda, 0, 0]^\rmJ$ ($\lambda \in M^{-1} \ZZ^N \cap \ZZ^N$) and $[0, \mu, 0]^\rmJ$ ($\lambda \in M^{-1} \ZZ^N \cap \ZZ^N$) on $\frake_0 \in V_{M, \alpha, \beta}$.  We find that
\begin{align*}
  [\lambda, 0, 0]^\rmJ \gamma\, \frake_0
=&
  \gamma [a \lambda, b \lambda, 0]^\rmJ\, \frake_0
=
  e(-a b M[\lambda])\,
  e( \rT \beta b \lambda ) e( \rT \alpha a \lambda)\, \frake_0
\\[4pt]
  [0, \mu, 0]^\rmJ \gamma\, \frake_0
=&
  \gamma [c \mu, d \mu, 0]^\rmJ\, \frake_0
=
  e(-c d M[\mu])\,
  e( \rT \beta d \mu ) e( \rT \alpha c \mu)\, \frake_0
\text{.}
\end{align*}
Since, by Lemma~\ref{la:integral-evaluation-of-M}, $M[\lambda]$ and $M[\mu]$ are integral, the proof of the first half of Lemma~\ref{la:permutation-of-heisenberg-irreps} is complete.

Since $\pi$ is unitary, the second half of Lemma~\ref{la:permutation-of-heisenberg-irreps} is a consequence of the simple observation that the system eigenvalues under $[0, \mu, 0]^\rmJ$ ($\mu \in \ZZ$) of eigenvectors in $V_{M, \alpha', \beta'}$ are pairwise different.
\end{proof}


The preceding lemma enables us to define minimal representations $\rho_{M, \alpha, \beta}$ of $\tGamma^{\rmJ(N)}$.  As a representation of $\rmH( \tGamma^{\rmJ(N)} )$, let
\begin{gather*}
  \rho_{M, \alpha, \beta} \big|_{\rmH( \tGamma^{\rmJ(N)} )}
=
  \bigoplus_{(\alpha', \beta') \in (\alpha, \beta) \SL{2}(\ZZ) \slashdiv \sim}
  \pi_{M, \alpha', \beta'}
\text{,}
\end{gather*}
where $(\alpha'_1, \beta'_1) \sim (\alpha'_2, \beta'_2)$ if and only if $(\alpha'_1 - \alpha'_2, \beta'_1 - \beta'_2) \in (M^\ZZ \ZZ^N)^2$.  Note that each $\rmH( \tGamma^{\rmJ(N)} )$-isomorphism class occurs at most once.  In accordance with Lemma~\ref{la:permutation-of-heisenberg-irreps}, we let $\gamma \in \Mp{2}(\ZZ)$ permute components $\pi_{M, \alpha', \beta'}$ by the right action of $\rT \gamma$ on $(\alpha', \beta')$.  By the same lemma, we can define the image of $v \in V_{M, \alpha, \beta}$ under $\gamma$ by the commutation relations in $\tGamma^{\rmJ(N)}$:
\begin{gather*}
  [\lambda, \mu, 0]^\rmJ \gamma\, v
=
  \gamma [a \lambda + c \mu, b \lambda + d \mu, 0]^\rmJ\, v
\text{.}
\end{gather*}

\begin{proposition}
\label{prop:tensor-decomposition-for-tGammaJ-representations}
Let $\rho$ be an irreducible representation.  Suppose that $\pi_{M, \alpha, \beta}$ occurs in $\rho |_{\rmH( \tGamma^{\rmJ(N)} )}$.  Then $\rho \cong \rho' \otimes \rho_{M, \alpha, \beta}$ for a unitary representation $\rho'$ of $\Mp{2}(\ZZ)$.
\end{proposition}
\begin{proof}
First consider the $\rmH( \tGamma^{\rmJ(N)} )$-isotypical decomposition of $\rho$.
\begin{gather*}
  \rho \big|_{\rmH( \tGamma^{\rmJ(N)} )}
\cong
  \sum_{(\alpha', \beta')} {\td \pi}_{M, \alpha', \beta'}
\text{.}
\end{gather*}
Because $\rho$ is irreducible and $\pi_{M, \alpha, \beta}$ occurs in $\rho |_{\rmH( \tGamma^{\rmJ(N)} )}$, Lemma~\ref{la:permutation-of-heisenberg-irreps} implies that $(\alpha', \beta')$ runs through $(\alpha, \beta) \SL{2}(\ZZ) \slashdiv \sim$.  Since, for the same reason, $\SL{2}(\ZZ)$ permutes the isotypical components transitively, we find that
\begin{gather*}
  \rho \big|_{\rmH( \tGamma^{\rmJ(N)} )}
\cong
  \pi_{\mathbbm{1}} \otimes \rho_{M, \alpha, \beta}
\text{,}
\end{gather*}
where $\pi_{\mathbbm{1}}$ a multiple of the trivial representation of $\rmH( \tGamma^{\rmJ(N)} )$.  By construction of $\rho_{M, \alpha, \beta}$, we have
\begin{gather*}
  \rho(\gamma) \, \big( \pi_{\mathbbm{1}} \otimes \rho_{M, \alpha, \beta}^{-1}(\gamma) \big) \;
  \rho\big( [\lambda, \mu, \kappa]^\rmJ \big)
=
  \rho\big( [\lambda, \mu, \kappa]^\rmJ \big) \;
  \rho(\gamma) \, \big( \pi_{\mathbbm{1}} \otimes \rho_{M, \alpha, \beta}^{-1} (\gamma) \big)
\end{gather*}
for all $\lambda, \mu \in \ZZ^N$, and $\kappa \in \Mat{N}(\ZZ)$.  Consequently, $\rho \cong \rho' \otimes \rho_{M, \alpha, \beta}$ is a tensor product.
\end{proof}

\subsection{Theta functions}

In this section, we set up the basic tool, the theta functions, to deduce the theta decomposition for Jacobi forms of arbitrary index.  Theta functions for indices in $\MatT{N}(\QQ)$ come with nontrivial representations of $\rmH( \tGamma^{\rmJ(N)} )$ attached to them.

We denote Fourier transforms by $\cF\cT$.  Recall that the Fourier transform of a function $f :\, \RR^N \rightarrow \CC$, by definition, is
\begin{gather*}
  \cF\cT(f)(\xi)
=
  \int_{-\infty}^\infty f(\xi')\, \exp(2 \pi i\, \rT \xi \xi') \;d\xi'
\text{.}
\end{gather*}

\begin{lemma}
\label{la:general-fourier-transform}
Let $K$ be a symmetric, positive definite matrix.  The Fourier transforms with respect to~$\xi'$ of
\begin{gather*}
  \exp\big(- \pi K[ \xi' + h ] \big)
\end{gather*}
equals
\begin{gather*}
  (\det K)^{-\frac{1}{2}} \exp(2 \pi i\, \rT h \xi)
  \exp\big( - \pi K^{-1}[\xi] \big)
\text{.}
\end{gather*}
\end{lemma}
\begin{proof}
The following equations are standard.
\begin{alignat*}{2}
  \cF\cT(f(\xi' + h))
&=
  \exp(2 \pi i \rT h \xi)\, \cF\cT(f(\xi'))(\xi)
&&\qquad
  \big(h \in \RR^N \big)
\text{,}
\\[2pt]
  \cF\cT(f(K \xi'))
&=
  \det K\cdot \cF\cT(f(\xi'))(\xi) \big|_{\xi \mapsto K^{-1}\xi}
&&\qquad
  \big(K \in \GL{N}(\RR) \big)
\text{,}
\\[2pt]
  \cF\cT(\exp(-\pi\, \rT \xi' \xi') )
&=
  \exp(-\pi \rT \xi \xi)
\text{.}
\end{alignat*}
By these rules, we have
\begin{align*}
  \cF\cT\big( \exp\big(- \pi K[ \xi + h ] \big) \big)
&=
  \exp(2 \pi i\, \rT h \xi)\,
  \cF\cT\big( \exp(-\pi K[\xi] ) \big)
\\
&=
  (\det K)^{-\frac{1}{2}}
  \exp(2 \pi i\, \rT h \xi)\,
  \cF\cT\big( \exp(- \pi\, \rT \xi \xi ) \big) \big|_{\xi \mapsto \sqrt{K}^{-1} \xi}
\\
&=
  (\det K)^{-\frac{1}{2}}
  \exp(2 \pi i\, \rT h \xi)\,
  \exp( -\pi K^{-1}[\xi] )
\text{.}
\end{align*}
\end{proof}

\begin{corollary}
\label{cor:theta-fourier-transform}
Given $\nu \in \QQ^N$, the Fourier transform of
\begin{gather*}
  \exp\big( \pi i \tau M^{-1} [M_\ZZ \xi + \nu]
            + 2 \pi i\, \rT z (M_\ZZ \xi + \nu) \big)
\end{gather*}
equals
\begin{multline*}
  \sqrt{\det M^{-1}[M_\ZZ]}^{-1}
  \exp\big(2 \pi i\, \xi M_\ZZ^{-1} \nu \big) \;
\\[2pt]
  \big( \frac{i}{\tau} \big)^{\frac{N}{2}}
  \exp\Big( \pi i\, \frac{-M[z]}{\tau} \Big) \;
  \exp\Big( \pi i\, \frac{-1}{\tau} M^{-1}[M \rT M_\ZZ^{-1} \xi] \Big)
  \exp\Big( 2 \pi i\, \rT(M \rT M_\ZZ^{-1} \xi) \frac{z}{\tau} \Big)
\text{.}
\end{multline*}
\end{corollary}
\begin{proof}
For $\tau = i y$ and $z = i v$ the result follows Lemma~\ref{la:general-fourier-transform} with $K = y M^{-1}[M_\ZZ]$ and $h = M_\ZZ^{-1} (\nu + \frac{1}{y} M v)$.  A straightforward calculation gives
\begin{multline*}
  \sqrt{\det M^{-1}[M_\ZZ]}^{-1}
  \exp\big(2 \pi i\, \xi M_\ZZ^{-1} \nu \big) \;
\\[2pt]
  y ^{-\frac{N}{2}}
  \exp\Big( \pi \frac{M[v]}{y} \Big) \;
  \exp\Big( - \pi y^{-1} M[\rT M_\ZZ^{-1} \xi] \Big)
  \exp\Big( 2 \pi i\, \rT(M \rT M_\ZZ^{-1} \xi) \frac{v}{y} \Big)
\text{.}
\end{multline*}
Since the first and second expression in the corollary's statement represent holomorphic functions, this proves the corollary.
\end{proof}

\begin{lemma}
\label{la:theta-series-MMZZ-relation}
We have
\begin{gather*}
  M \rT M_\ZZ^{-1} \ZZ^N
=
  M \ZZ^N + \ZZ^N
\text{.}
\end{gather*}
\end{lemma}
\begin{proof}
We proceed as in the proof of Lemma~\ref{la:integral-M-and-its-inverse}.  Recall that we have $M = U D V$ and $M_\ZZ = U D_\ZZ V$.  In order to simplify computations, we use the fact that if $M' = M[g]$ for $g \in \GL{N}(\ZZ)$, then $M' (\rT M'_\ZZ)^{-1} = \rT g M \rT M_\ZZ^{-1} \rT g^{-1}$.  That is, we may assume that $V$ is the identity matrix.  Because $M$ is symmetric, we have $U D = D \rT U$.  Now, the following computation proves the lemma.
\begin{gather*}
  M \rT M_\ZZ^{-1} \ZZ^N
=
  U D \rT U^{-1} D_\ZZ^{-1} \ZZ^N
=
  D D_\ZZ^{-1} \ZZ^N
=
  M \ZZ^N + \ZZ^N
\text{.}
\end{gather*}
\end{proof}

We call $\disc\, M = (M \ZZ^N + \ZZ^N) \slashdiv (M \ZZ^N \cap \ZZ^N)$ the (generalized) discriminant form associated to $M$.  Note that is not a discriminant form in the sense of~\cite{Ni79}, since the associated quadratic form $q_M(\lambda) = M^{-1}[\lambda]$ does not necessarily take integral values on \mbox{$M \ZZ^N \cap \ZZ^N$}.  Let $\CC[\disc\,M]$ be the group algebra of $\disc\, M$, which has basis $\frake_\nu$ ($\nu \in \disc\,M$).

For $\nu \in \disc\, M$, define component functions
\begin{gather}
  \theta_{M, \nu} (\tau, z)
=
  \sum_{\xi \equiv \nu \pmod{M \ZZ^N \cap \ZZ^N}}
  q^{\frac{1}{2} M^{-1}[\xi]} \zeta^\xi
\text{,}
\end{gather}
where $\xi$ runs through $M \ZZ^N \cap \ZZ^N$.  We find that
\begin{gather*}
  \theta_M (\tau, z)
=
  \sum_{\nu \in (M \ZZ^N + \ZZ^N) \slashdiv (M \ZZ^N \cap \ZZ^N)} \theta_{M, \nu}(\tau, z) \, \frake_\nu
\text{,}
\end{gather*}
as a function $\HS^{\rmJ(N)} \rightarrow \CC[\disc\, M]$, transforms as a vector valued modular form under~$\Mp{2}(\ZZ)$.

Let $\rho_M$ be the representation of $\tGamma^{\rmJ(N)}$ on $\GL{}(\CC[\disc\, M])$, whose central character is $e(\frac{1}{2} \tr(\cdot M))$ and which is defined as
\begin{align*}
  \rho_M( S )\, \frake_\nu
&=
  \sqrt{\# \disc\, M}^{-1} \sum_{\nu' \in \disc\, M}
                 \exp\big(2 \pi i\, \rT \nu M^{-1} \nu' \big)\, \frake_{\nu'}
\text{,}
\\
  \rho_M( T )\, \frake_\nu
&=
  \exp\big( \pi i M^{-1} [\nu] \big)\, \frake_\nu
\text{,}
\\
  \rho_M ( [\lambda, 0, 0]^\rmJ)\, \frake_\nu
&=
  \frake_{\nu + M \lambda}
\text{,}
\quad\text{and}
\\
  \rho_M ( [0, \mu, 0]^\rmJ )\,\frake_\nu
&=
  \exp(2 \pi i \rT \nu \mu )\, \frake_\nu
\text{,}
\end{align*}
where we identify $S$ and $T$ with the corresponding elements in $\tGamma^{\rmJ(N)}$.

\begin{theorem}
\label{thm:theta-functions-for-M}
For any positive definite $M \in \MatT{N}(\QQ)$, the theta function $\theta_M$ is a vector valued Jacobi form of weight~$\frac{N}{2}$, index~$M$, and type $\rho_M$. 
\end{theorem}
\begin{proof}
We have to prove that
\begin{gather*}
  \theta_M \big|_{\frac{N}{2}, M}\, \gamma^\rmJ
=
  \rho_M(\gamma^\rmJ) \theta_M
\end{gather*}
for all $\gamma^\rmJ \in \{ T, S, [\lambda, 0, 0]^\rmJ, [0, \mu, 0]^\rmJ, [0, 0, \kappa]^\rmJ \}$.  All cases but $\gamma^\rmJ = S$ are clear.  In order to check the case~$\gamma^\rmJ = S$, we apply Poisson summation and Corollary~\ref{cor:theta-fourier-transform}.  Since we have $M_\ZZ \ZZ^N = M \ZZ^N \cap \ZZ^N$ by Lemma~\ref{la:integral-M-and-its-inverse}, we find
\begin{align*}
  \theta_{M, \nu}(\tau, z)
&=
  \sum_{\xi \in \ZZ^N}
   \exp\big( \pi i\, \tau M^{-1}[M_\ZZ \xi + \nu] + 2 \pi i\, \rT v (M_\ZZ \xi + \nu) \big)
\\
&=
  \sqrt{\det M^{-1}[M_\ZZ]}^{-1}
  \sum_{\xi' \in \ZZ^N }
  \exp\big(2 \pi i\, \xi' M_\ZZ^{-1} \nu \big) \;
  \big( \frac{i}{\tau} \big)^{\frac{N}{2}}
  \exp\Big( \pi i\, \frac{-M[z]}{\tau} \Big)
\\
&\qquad\quad
  \cdot
  \exp\Big( \pi i\, \frac{-1}{\tau} M^{-1}[M \rT M_\ZZ^{-1} \xi']
            + 2 \pi i\, \rT(M \rT M_\ZZ^{-1} \xi') \frac{z}{\tau} \Big)
\end{align*}
By Lemma~\ref{la:theta-series-MMZZ-relation}, we have $M \rT M_\ZZ^{-1} \ZZ^N = M \ZZ^N + \ZZ^N$.  This allows us to split the sum over $\xi'$ with respect to congruence classes in $(M \ZZ^N + \ZZ^N) \slashdiv (M \ZZ^N \cap \ZZ^N)$.  We then set $\xi = M \rT M_\ZZ^{-1} \xi'$, and obtain the result. 
\end{proof}

By Proposition~\ref{prop:tensor-decomposition-for-tGammaJ-representations}, we have
\begin{gather}
  \rho_M
\cong
  \rho^{\,\prime}_M \otimes \rho_{M, 0, 0}
\end{gather}
for a suitable representation $\rho^{\,\prime}_M$ of $\Mp{2}(\ZZ)$.  We shift the elliptic variable~$z$, to obtain representations $\rho^{\,\prime}_M \otimes \rho_{M, \alpha, \beta}$ for all $\alpha, \beta \in \QQ^N$.  Given $\alpha, \beta \in \QQ^N$ and $\nu \in (M \ZZ^N + \ZZ^N) \slashdiv (M \ZZ^N \cap \ZZ^N)$, define
\begin{gather*}
  \theta_{M, \nu}^{(\alpha, \beta)}
=
  \theta_{M, \nu}(\tau, z) \big|_{k, M} [M^{-1} \alpha, 0, 0]^\rmJ [0, M^{-1} \beta, 0]^\rmJ
\end{gather*}
and
\begin{align}
  \theta_M^{(\alpha, \beta)}
&:\,
  \HS^{\rmJ(N)} \rightarrow \CC[\disc\,M ] \otimes \CC\big[ (\alpha, \beta) \SL{2}(\ZZ) \slashdiv \sim \big]
\\\nonumber
  \theta_{M}^{(\alpha, \beta)}(\tau, z)
&=
  \sum_{\substack{ (\alpha', \beta') \in (\alpha, \beta) \SL{2}(\ZZ) \slashdiv \sim \\
                   \nu \in (M \ZZ^N + \ZZ^N) \slashdiv (M \ZZ^N \cap \ZZ^N) }}
  \theta_{M, \nu}^{(\alpha', \beta')}\, \frake_\nu \otimes \frake_{\alpha', \beta'}
\text{,}
\end{align}
where the sum runs over representatives modulo~$(M^\ZZ \ZZ^N)^2$, and $\frake_{\alpha', \beta'}$ is a basis element of $\CC\big[ (\alpha, \beta) \SL{2}(\ZZ) \slashdiv \sim \big]$.

\begin{proposition}
For any $\alpha, \beta \in \QQ^N$, $\theta_{M}^{(\alpha, \beta)}$ is a vector valued Jacobi form of weight~$\frac{N}{2}$, index~$M$, and type~$\rho^{\,\prime}_M \otimes \rho_{M, \alpha, \beta}$.
\end{proposition}
\begin{proof}
In conjunction with Propositon~\ref{prop:tensor-decomposition-for-tGammaJ-representations} and Theorem~\ref{thm:theta-functions-for-M}, it suffices to check the action of $\Mp{2}(\ZZ)$ and $[\lambda, \mu, 0]^\rmJ$ ($\lambda, \mu \in M\ZZ^N \cap \ZZ^N$) on 
\begin{gather*}
  \theta_{M}(\tau, z) \big|_{k, M} [M^{-1} \alpha', 0, 0]^\rmJ [0, M^{-1} \beta', 0]^\rmJ
\text{.}
\end{gather*}
A straightforward computation yields the claim.
\end{proof}

\subsection{Theta decomposition}

The aim of this section to deduce theta decomposition for all Jacobi forms of index $0 < M \in \MatT{N}(\QQ)$.  We start by limiting the possible action of the center $\Mat{N}(\ZZ)$ of $\tGamma^{\rmJ(N)}$.
\begin{proposition}
\label{prop:kappa-eigenvalues}
Suppose that $\rho$ is irreducible.  If
\begin{gather*}
  \rho\big( [0,0,\kappa]^\rmJ \big)
\ne
  e\big( \frac{1}{2} \tr( M \kappa ) \big)
\text{,}
\end{gather*}
then $\rmJ_{k, M}(\rho) = \{ 0 \}$.
\end{proposition}
\begin{proof}
Any Jacobi form satisfies
\begin{gather*}
  \phi
=
  \phi \big|_{k, M, \rho}\, [0, 0, \kappa]^\rmJ
=
  e( \tfrac{1}{2} \tr(M \kappa) ) \rho( [0, 0, -\kappa]^\rmJ ) \,
  \phi
\text{.}
\end{gather*}
This implies the statement straightforwardly.
\end{proof}

Recall that any Jacobi form $\phi$ has Fourier expansion
\begin{gather}
\label{eq:jacobi_fourierexpansion}
  \phi(\tau, z)
=
  \sum_{m \in \QQ,\, r \in \QQ^N} c(\phi; m, r)\, q^m \zeta^r
\text{,}
\end{gather}
where $q = e(\tau)$ and $\zeta^r = e( \rT r z )$.  The Fourier coefficients $c(\phi; m, r)$ vanish, if $2 \det(M) m - M^\#[r] < 0$.

We write ${\check \rho}$ for the dual of any representation $\rho$.
\begin{theorem}
\label{thm:theta-decomposition}
Let $\phi$ be a Jacobi form of irreducible type~$\rho$.  Then there is $\alpha, \beta \in \QQ^N$ such that
\begin{enumerate}[(i)]
\item There is a unitary, irreducible representation~$\rho'$ of $\Mp{2}(\ZZ)$ with $\rho \cong \rho' \otimes \rho_{M, \alpha, \beta}$.

\item We have
\begin{gather}
\label{eq:theta-decomposition}
  \phi(\tau, z)
=
  \sum_{\nu \in (M \ZZ^N + \ZZ^N) \slashdiv M \ZZ^N}
  h_\nu(\tau)
  \sum_{\substack{ (\alpha', \beta') \in (\alpha, \beta) \SL{2}(\ZZ) \\
                   \nu' \in M \ZZ^N \slashdiv (M \ZZ^N \cap \ZZ^N) }}
  \theta_{M, \nu + \nu'}^{(\alpha', \beta')}(\tau, z)\,
  \frake_{\nu + \nu'} \otimes \frake_{\alpha', \beta'}
\end{gather}
for some $(h_\nu) \in \rmM_{k - \frac{N}{2}}(\rho' \otimes {\check \rho}^{\, \prime}_M)$.

Conversely, given a vector valued modular form of weight~$k - \frac{N}{2}$ and type~$\rho \otimes {\check \rho}^{\, \prime}_M$, Equation~(\ref{eq:theta-decomposition}) defines an element of~$\rmJ_{k, M}(\rho \otimes \rho_{M, \alpha, \beta})$.
\end{enumerate}
\end{theorem}
\begin{proof}
Note that the sum $\nu + \nu'$, which occurs in the subscript, is well-defined, since $\nu'$ runs through all representatives modulo~$M \ZZ^N \cap \ZZ^N$. 

By Proposition~\ref{prop:tensor-decomposition-for-tGammaJ-representations}, there is a representation $\rho'$ such that $\rho \cong \rho' \otimes \rho_{M, \alpha, \beta}$ for some $\alpha, \beta \in \QQ^N$.  Since $\rho$ is irreducible, $\rho'$ must be irreducible, too.  We can assume that $\rho = \rho' \otimes \rho_{M, \alpha, \beta}$, thus identifying representation spaces.

The Fourier expansion of $\phi$ is uniquely determined by $c(\phi; n, r)$, where $r$ runs through a system of representatives of $(M \ZZ^N + \ZZ^N) \slashdiv M \ZZ^N$.  Fix some lift $r$ of $\nu$ to $M \ZZ^N + \ZZ^N$ and one pair $(\alpha, \beta)$.  It is straightforward to check corresponding Fourier coefficients $c(\,\cdot\,; n, r)$ of the left and right hand side.  This proves the first part of the theorem.

The second part follows directly, when checking transformation properties.
\end{proof}
We denote the induced bijection between $\rmJ_{k, M}(\rho)$ and $\rmM_{k - \frac{N}{2}}({\check \rho}^{\, \prime}_M \otimes \rho')$ by $\Theta_M$, suppressing the weight.

\subsection{Vanishing of Jacobi forms}
\label{ssec:jacobi-forms:vanishing}

We deduce vanishing statements for Jacobi forms based on their connection to vector valued elliptic modular forms.  The theta decomposition allows us to do this.

Let
\begin{gather*}
  \rd(M)
=
  \max_{\xi \in \RR^N} \min_{\xi' \in \ZZ^N} M[\pm \xi + \xi']
\text{.}
\end{gather*}
Note that $\rd(M)$ is related to the radius of a Voronoi cell of a lattice with Gram matrix~$M$.

\begin{proposition}
\label{prop:vanishing-of-jacobi-forms}
If $\phi \in \rmJ_{k, M}(\rho)$ and $c(\phi; m, r) = 0$ for all $0 \le m < \frac{k}{12} + 1 + \frac{1}{2} \rd(M)$ and all $r \in M \ZZ^N + \ZZ^N$.  Then $\phi = 0$.
\end{proposition}
\begin{proof}
We can assume that $\rho$ is irreducible.  According to Theorem~\ref{thm:theta-decomposition}, we have $\rho \cong \rho' \otimes \rho_{M, \alpha, \beta}$ for some $\alpha, \beta \in \QQ^N$.

We apply Proposition~\ref{prop:vanishing-of-vector-valued-forms} to $\Theta_M(\phi)$.  Given $\nu \in (M \ZZ^N + \ZZ^N) \slashdiv M \ZZ^N$ and a lift~$r \in \QQ^N$ of~$\nu$, we have
\begin{gather*}
  c\big( \Theta_M(\phi)_\nu; m)
=
  \big( c(\phi; m + \tfrac{1}{2} M^{-1}[r + \alpha], r + \alpha) \big)_{r;\, \alpha, \beta}
\text{,}
\end{gather*}
where the subscripts $r$ and $\alpha, \beta$ refer to the corresponding coordinates in $\CC[\disc\,M]$ and $\CC\big[ (\alpha, \beta) \SL{2}(\ZZ) \slashdiv \sim \big]$, respectively.

Combining this with Relation~(\ref{eq:higherjacobiforms_fourier_relation2}), we see that it suffices to prove that for every $\nu \in (M \ZZ^N + \ZZ^N) \slashdiv M \ZZ^N$, there is $\xi \in \pm \nu \subset \QQ^N$ with $M^{-1}[\xi] \le \rd(M)$.  In order to connect $\nu$ to the definition of $\rd(M)$, we consider $M^{-1} \nu \in (M^{-1} \ZZ^N + \ZZ^N) \slashdiv \ZZ^N$ instead of $\nu$.  We have to find $ M^{-1} \xi \in \pm M^{-1} \nu$ with $M^{-1}[ \xi ] = M[ M^{-1} \xi ] \le \rd(M)$.  By definition of $\rd(M)$, this is possible, since $M^{-1} \xi$ is defined up to elements in $\ZZ^N$.  This finishes the proof.
\end{proof}

Given $0 < M \in \MatT{N}(\QQ)$, define
\begin{gather*}
  \md(M)
=
  \min_{g \in \GL{N}(\ZZ)} \max_{1 \le i \le N}(M[g]_{i,i})
\text{.}
\end{gather*}
This is the length of the longest vector within a basis of shortest vectors.

\begin{proposition}
\label{prop:relation-of-rd-and-md}
Given $0 < M \in \MatT{N}(\QQ)$, we have $\rd(M) \le \frac{N (N + 1)}{8} \, \md(M)$.
\end{proposition}
\begin{proof}
Since the proposition holds for $M[g]$, $g \in \GL{N}(\ZZ)$ if it holds for $M$, we can assume that $M_{i,i} \le M_{i + 1, i + 1}$ for $1 \le i < N$.  Further, we can assume that for all $1 \le i < N$, we have
\begin{gather*}
  M_{i,i}
=
  \min\big\{ M[g]_{i,i} \,:\, M[g]_{j,j} = M_{j,j} \text{ for $1 \le j < i$},\,
                              g \in \GL{N}(\ZZ) \big\}
\text{.}
\end{gather*}
In particular, we have $\md(M) = M_{N, N}$.  Moreover, we find that $|M_{i,j}| \le \frac{1}{2} \min\{ M_i, M_j \}$.  Indeed, suppose that this was not the case for some $i < j$.  Then we could replace the basis vector $e_j$ by $e_i \pm e_j$.  We would then have
\begin{gather*}
  M[ e_i \pm e_j ]
=
  M_i + M_j \pm 2 M_{i,j}
<
  M_j
\end{gather*}
for the right choice of sign.  This would contradict the minimality of diagonal elements of $M$.

Given any $\xi \in \RR^N$, we can find $w = \pm \xi + \xi'$ for some choice of sign and $\xi' \in \ZZ^N$ such that $|\pm \xi_i + \xi'_i| \le \frac{1}{2}$ for $1 \le i \le N$.  We then have
\begin{align*}
  \rd(M)
\le
  M[w]
& \le
  \sum_{1 \le i, j \le N} |M_{i,j}| |w_i| |w_j|
\le
  \frac{1}{4} \sum_{1 \le i, j \le N} |M_{i,j}|
\\[3pt]
& \le
  \frac{1}{4} \sum_{ 1 \le i \le j \le N } |M_{N, N}|
=
  \frac{N (N + 1)}{8} \, \md(M)
\text{.}
\end{align*}
\end{proof}
\begin{remark}
In the next section, we use $\rd(M) < (2 - \epsilon) \md(M)$ for $N < 4$.  In the case $N = 4$, the diagonal Gram matrix $M = {\rm diag}(2, 2, 2, 2)$ has $\md(M) = 2$ and $\rd(M) = 1$.
\end{remark}

\section{Special cycles}
\label{sec:special-cycles}

As announced in the introduction, we switch notation and adopt Zhang's.

Let $\cL$ be an integral, even lattice of signature $(n, 2)$, and write $\cL^\#$ for its dual.  The Grassmannian of $2$-dimensional negative subspaces of $\cL \otimes \RR$ is denoted by ${\rm Gr}^-(\cL \otimes \RR)$.  Fix a subgroup~$\Gamma$ of $\Orth{}(\cL)$ that act trivially on $\disc\, \cL := \cL^\# \slashdiv \cL$.  The quotient $X_\Gamma = \Gamma \backslash {\rm Gr}^-(\cL \otimes \CC)$ is a Shimura variety of orthogonal type, on which there are rational quadratic cycles.  Given a tuple of vectors $v = (v_1, \ldots, v_r) \in (\cL \otimes \QQ)^r$, let $Z(v) = \{ W \in {\rm Gr}^-(\cL \otimes \CC) \,:\, \lspan(v) \perp W \}$.  This cycle is non-trivial, if $\lspan(v) \subseteq \cL \otimes \QQ$ is positive.

Kudla~\cite{Ku97} built so-called special cycles from the $Z(v)$.  To each $v$, we attach a moment matrix $q_\cL(v) = \frac{1}{2} ( \langle v_i, v_j \rangle_\cL )_{1 \le i,j \le r}$, where $\langle v, w \rangle_\cL = q_\cL(v + w) - q_\cL(v) - q_\cL(w)$.  Given a positive semidefinite matrix $0 \le T \in \MatT{r}(\QQ)$ and $\lambda \in (\cL^\# \slashdiv \cL)^r$, set
\begin{gather*}
  \Omega(T, \mu)
=
  \big\{ v \in \mu + \cL^r \,:\, T = q_\cL(v) \big\}
\text{.}
\end{gather*}
The symmetries $\Gamma$ act on $\Omega(T, \mu)$, and there are finitely many orbits.  Kudla defines special cycles as
\begin{gather*}
  Z(T, \mu)
=
  \sum_{v \in \Gamma \backslash \Omega(T, \mu)}
  Z(v)
\text{.}
\end{gather*}
All cycles of this form descend to cycles on $X_\Gamma$, which we also denote by $Z(T, \mu)$.  Writing $\rk(T)$ for the rank of~$T$, we find that $Z(T, \mu)$ is a cycle of codimension~$\rk(T)$, whose class in $\CH^{\rk(T)}(X_\Gamma)_\CC$ is denoted by $\{Z(T, \mu)\}$.

The aim of this section is to prove the next theorem.
\begin{theorem}
\label{thm:main-theorem-finite-span-of-cycles}
Suppose that $\Gamma \subset \Orth{}(\cL)$ acts trivially on $\disc\, \cL$.  Then for $r < 5$, the space
\begin{gather*}
  \lspan\big( \{Z(T, \mu)\} \,:\, 0 < T \in \MatT{r}(\QQ),\, \mu \in \disc^r\,\cL \big)
\subseteq
  \CH^r(X_\Gamma)_\CC
\end{gather*}
is finite dimensional.
\end{theorem}

Given $M \in \MatT{r - 1}(\QQ)$, Zhang formed the partial generating series 
\begin{gather}
\label{eq:def:theta-series}
  \theta_{\mu, \mu', M}
=
  \sum_{\substack{ 0 \le n \in \ZZ \\ p \in \ZZ^{r - 1} }}
  Z(T, (\mu, \mu')) \, q^n \zeta^p
\text{,}
\end{gather}
where $T = \left(\begin{smallmatrix} n & \frac{1}{2} \rT p \\ \frac{1}{2} p & M \end{smallmatrix}\right)$.

Based on results in~\cite{Bo99, Bo00}, Zhang proved in his thesis~\cite{Zh09} that $\theta_{\mu, \mu', M}$ is a vector valued Jacobi form of weight $1 + \frac{n}{2}$.  For the convenience of the reader, we restate his result.
\begin{proposition}[{\cite{Zh09}, Proposition~2.6 and (2.10)}]
\label{prop:zhangs-modularity}
The right hand side of~(\ref{eq:def:theta-series}) is absolutely convergent, and we have
\begin{gather*}
  \sum_{\substack{ \mu \in \disc\, \cL \\ \mu' \in (\disc\, \cL)^{r -1} }}
  \theta_{\mu, \mu', M} \,
  \frake_{\mu}
\in
  \rmJ_{1 + \frac{n}{2}, M}(\rho_\cL)
\text{,}
\end{gather*}
where $\rho_{\cL, 2}$ is the Weil representation on $\Sp{2}(\ZZ)$ associated to~$\cL$ (see~\cite{Zh09}).
\end{proposition}

\begin{proposition}
\label{prop:span-with-min-bound}
Given any $0 < B \in \RR$, the space
\begin{align*}
  \lspan\Big( \{ Z\big(T, (\mu, \mu')\big) \} \,:\;\,
&
              T = \left(\begin{smallmatrix} m & \frac{1}{2} \rT p \\ \frac{1}{2} p & M \end{smallmatrix}\right) \in \MatT{r}(\QQ),\,
              \mu \in \disc\,\cL, \mu' \in (\disc\,\cL)^{r - 1},
\\[2pt]
&
              m, \md(M) < B + \tfrac{1}{2} \rd(M)
  \Big)
\end{align*}
is finite dimensional, if $r < 5$.
\end{proposition}
\begin{proof}
Because of the $\GL{r}(\ZZ)$ invariance of $Z(T, (\mu, \mu'))$ and because the denominator of $q_\cL(\lambda)$, $\lambda \in \mu$ is bounded, it is sufficient to consider cycles associated to $\GL{r-1}(\ZZ) \ltimes \ZZ^{r-1}$ classes
\begin{gather*}
 \Big\{ \left(\begin{smallmatrix} m                      & \rT (\rT g p + \lambda M[g]) \\
                                  \rT g p + \lambda M[g] & M[g]
        \end{smallmatrix}\right)
   \,:\, g \in \GL{r - 1}(\ZZ),\, \lambda \in \ZZ^{r - 1}
 \Big\}
\end{gather*}
of positive matrices with bounded denominator.

First of all there are only finitely many $\GL{r-1}(\ZZ)$ classes of matrices $M$ with bounded denominator that satisfy $\md(M) < B + \frac{1}{2}\rd(M)$.  This follows from Proposition~\ref{prop:relation-of-rd-and-md}, and the fact that, for any $B'$, there are only finitely many $\GL{r - 1}(\ZZ)$ classes with bounded denominator and $\md(M) < B'$.

Next, we can assume that $p$ is reduced with respect to the column space of~$M$, so that there are only finitely many $p$ that can occur for any given~$M$.  Only finitely many $m$ occur for given $M$ and $p$, since $0 < m < B + \frac{1}{2}\rd(M)$.  This completes the proof.
\end{proof}

\begin{theorem}
\label{thm:generators-for-spans-of-special-cycles}
We have
\begin{multline*}
  \lspan\big( \{Z(T, \mu)\} \,:\, 0 < T \in \MatT{r}(\QQ),\, \mu \in (\disc\,\cL)^r \big)
\\[3pt]
=
  \lspan\Big( \{Z(T, \mu)\} \,:\,
              0 < T = \left(\begin{smallmatrix} m & \frac{1}{2} \rT p \\
                                                \rT p & M \end{smallmatrix}\right)
                      \in \MatT{r}(\QQ),\,
              \mu \in (\disc\,\cL)^r,\;
\\
              m, \md(M) < 1 + \frac{2 + 2 n}{24} + \tfrac{1}{2} \rd(M)
        \Big)
\text{.}
\end{multline*}
\end{theorem}
\begin{proof}
Assume that the above is not true.  Then there is a functional~$f$ on $\CH^r(X_\Gamma)_\CC$ that vanishes on the right hand side but not on the left hand side.  Write $f(\theta_{\mu, \mu', M})$ for
\begin{gather*}
  \sum_{\substack{ 0 \le m \in \QQ \\ p \in \QQ^{r - 1} }}
  f\big( Z(T, (\mu, \mu')) \big) \, q^m \zeta^p
\text{.}
\end{gather*}
By choice of $f$, we have $c\big( f(\theta_{\mu, \mu', M}); m, p \big) \ne 0$ for some $m, p$, and $M$.  We show that this cannot be.

Set $B = 1 + \frac{2 + 2m}{24}$.  We consider $M$ with $\md(M) \le B + \frac{1}{2} \rd(M)$.  Then $c\big( f(\theta_{\mu, \mu', M}); m, p \big)$ vanishes by assumption for all $m < 1 + \frac{2 + 2 n}{24} + \rd(M)$.  By Theorem~\ref{prop:zhangs-modularity} and Proposition~\ref{prop:vanishing-of-jacobi-forms}, we find that $f(\theta_{\mu, \mu', M}) = 0$.

Next, we use induction on $(M_{i,i})_{1 \le i \le N} \in \ZZ^N$ to show that $f(\theta_{\mu, \mu', M}) = 0$.  We impose the componentwise comparison on $\ZZ^N$.  In other words, for fixed $M$ with $\md(M) > B + \frac{1}{2} \rd(M)$, we suppose that $f(\theta_{\mu_2, \mu'_2, M_2}) = 0$, whenever $(M_2)_{i,i} \le M_{i,i}$ for all $1 \le i \le N$, and there is one $1 \le i \le N$ such that $(M_2)_{i,i} < M_{i,i}$.

We can make the same assumptions as in the proof of Proposition~\ref{prop:relation-of-rd-and-md}.  That is, we can assume that $M_{i,i} < M_{i + 1, i + 1}$ for $1 \le i < N$, and for all $1 \le i \le N$
\begin{gather*}
  M_{i,i}
=
  \min\big\{ M[g]_{i,i} \,:\, M[g]_{j,j} = M_{j,j} \text{ for $1 \le j < i$},\,
                              g \in \GL{N}(\ZZ) \big\}
\text{.}
\end{gather*}

We want to show that $f(\theta_{\mu, \mu', M})$ vanishes using Proposition~\ref{prop:vanishing-of-jacobi-forms}.  We must argue that $c\big( f(\theta_{\mu, \mu', M}); m, p \big) = 0$ for all $m < B + \frac{1}{2} \rd(M) < \md(M)$.  For any such $m$ we have
\begin{gather*}
  c\big( f(\theta_{\mu, \mu',\, M}); m, p \big)
=
  c\big( f(\theta_{\mu'_{r - 1}, (\mu'_1, \ldots, \mu'_{r - 2}, \mu),\, M'}); M_{N, N}, p' \big)
\text{,}
\end{gather*}
where $M'$ has diagonal elements $M'_{i,i} = M_{i,i}$ if $i \ne N$, and $M'_{N,N} = m < \md(M) = M_{N, N}$.  Thus we can employ the induction hypothesis to prove that the right hand side is zero.  This way, we establish the hypothesis of Proposition~\ref{prop:vanishing-of-jacobi-forms}, and finish the induction.

We have shown $f( \theta_{\mu, \mu', M} ) = 0$ for all $M$.  This contradicts the choice of $f$, and finishes the proof.
\end{proof}

\begin{proof}[{Proof of Theorem~\ref{thm:main-theorem-finite-span-of-cycles}}]
Combine Theorem~\ref{thm:generators-for-spans-of-special-cycles} and Proposition~\ref{prop:span-with-min-bound}.
\end{proof}

\bibliographystyle{amsalpha}
\bibliography{bibliography}

\end{document}